\documentclass[final]{siamltex}
\usepackage{amssymb, amsmath}
\usepackage{float,epsfig}
\usepackage{algpseudocode}
\usepackage{color}
\usepackage{multirow}
\usepackage{cleveref}
\usepackage{braket}
\usepackage{subfigure}
\usepackage[utf8]{inputenc}
\usepackage{tikz}
\usepackage{enumitem}
\usetikzlibrary{patterns}
\usetikzlibrary{graphs}
\usepackage{epsfig}
\usepackage{hyperref}
\usepackage{graphicx}
\usepackage{graphics}
\usepackage{tikz-network}
\usepackage{caption}
\usepackage{subcaption}
\usetikzlibrary{graphs,graphs.standard,calc}

\newcommand{\Arrow}[1]{%
	\parbox{#1}{\tikz{\draw[->](0,0)--(#1,0);}}
}

\newcommand{\rsss}{\rotatebox[]{90}{$\boxminus$}\kern-0.7em{\mathrel{\raisebox{.2ex}{\Arrow{.35cm}}}}\!\!}

\newcommand{\csss}{\text{$\boxminus\kern-0.655em{\mathrel{\raisebox{-.2ex}{\rotatebox[]{-90}{\Arrow{.34cm}}}}}$\ }}

\newcommand{\ba}{\begin{array}}
\newcommand{\ea}{\end{array}}

\newtheorem{example}[theorem]{\bf Example}
\newtheorem{note}{Note}
\newcommand{\be}{\begin{equation}}
\newcommand{\ee}{\end{equation}}
\newcommand{\beano}{\begin{eqnarray*}}
\newcommand{\eeano}{\end{eqnarray*}}

\def \A{{\mathcal A}}
\def \N{{\mathbb N}}

\def \d{{\bf d}}
\def \r{{\bf r}}

\def \diag{\mathrm{diag}}

\def \ker{\mathrm{Ker}}
\def \Im{\mathsf{Im}}

\title{Arithmetical Structures on Wheel Graphs}
\author{Bibhas Adhikari \thanks{Department of Mathematics, IIT Kharagpur, Kharagpur-721302, India, ({\tt bibhas@maths.iitkgp.ac.in}). The author currently works at Fujitsu Research of America, Inc., USA} \and Namita Behera \thanks{Department of Mathematics, Sikkim University, Sikkim-737102, India, ({\tt nbehera@cus.ac.in}, niku.namita@gmail.com)} \and  Dilli Ram Chhetri \thanks{Department of Mathematics, Sikkim University, Sikkim-737102, India, ({\tt drchhetri.22pdmt01@sikkimuniversity.ac.in}).} 
\and Raj Bhawan Yadav \thanks{Department of Mathematics, Sikkim University, Sikkim}-737102, India, ({\tt rbyadav01@cus.ac.in}) }

\begin{document}
\maketitle
\begin{abstract} An arithmetical structure on a finite and connected graph $G$ is a pair $(\textbf{d}, \textbf{r})$ of positive integer vectors such that $\textbf{r}$ is primitive (the gcd of its entries is 1) and $(\diag(\textbf{d}) - A)\textbf{r} = 0,$ where $A$ is the adjacency matrix of $G.$ In this article, we investigate arithmetical structures on the wheel graphs. 
\end{abstract}

\begin{keywords}
Wheel graphs, Laplacian, Arithmetical structures, M-matrices    
\end{keywords}

\begin{AMS}
 11D72, 15B48, 11D45, 11B83, 11D68, 11C20   
\end{AMS}







\section{Introduction}
An arithmetical structure on a graph is a combinatorial construct that associates integer-valued functions to the vertices of the graph, satisfying certain linear algebraic conditions related to the graph's adjacency matrix. Introduced by Lorenzini \cite{lorenzini1991finite}, the notion of arithmetical structures arise in the study of degenerating curves in algebraic geometry. Indeed, an arithmetical graph is a triplet $(G, {\bf d}, {\bf r})$ given by a graph $G=(V,E)$ and a pair of vectors $\d, \r\in \N^{V}$ such that gcd$\left(r_v: v\in V\right)=1$ and \begin{equation}\label{eqn:vcond}
    d_vr_v=\sum_{u\in \, \mathcal{N}_v} r_u, \, v\in V,
\end{equation} 
where $\mathcal{N}_v$ denotes the neighbourhood of a vertex $v$ and $\N$ is the set of positive integers. Given an arithmetical graph $(G, \d, \r),$ we say that the pair $(\d, \r)$ is an arithmetical
structure of $G$. A well-known result from \cite{lorenzini1991finite} proves that there exist only finitely many arithmetical structures on any connected simple graph. Consequently, given a graph $G,$ the set of arithmetical structures is defined as \begin{equation}\label{eqn:asofg}
    \mathcal{A}(G)=\left\{ (\d,\r)\in \N^V\times \N^V : (\d , \r ) \,\, \mbox{is an arithmetical structure of} \,\, G \right\},
\end{equation} and it is a paramount interest in literature to characterize $\mathcal{A}(G)$ for a given $G.$ Since each $\d$ (resp $\r$) determines $\r$ (resp. $\d$) uniquely, any one of $\d$, $\r$ or $(\d,\r)$ may be considered as arithmetical structures of the graph. Following the literature, we use the  terms $\d$-structure and $\r$-structure instead of arithmetical structure when we consider only $\d$ and $\r,$ respectively \cite{BHDNJC18}. We denote the set of $\r$-structures of $G$ as $\A_{\r}(G).$


Arithmetical structures have deep connections with number theory, algebraic geometry, and potential theory on graphs, as they generalize concepts like divisor theory on curves and abelian sandpile models \cite{biggs1997algebraic,corrales2018arithmetical}. The study of arithmetical structures on graphs also intersects with topics like chip-firing games and the study of critical groups (Jacobians) of graphs, which have applications in various fields, including network theory and statistical physics \cite{baker2007riemann,Corry2018DivisorsAS}.

Before proceeding, let us recall some graph theory terminologies following \cite{diestel2024graph}. Let $G=(V,E)$ denote a simple connected graph with $V$ is the vertex set and $E\subset V\times V$ represents the edge set. In this article, all graphs are undirected and finite. The complete graph on $n$ vertices is denoted by $K_n,$ and the cycle graph on $n$ vertices is denoted by $C_n.$ The adjacency matrix $A=[a_{ij}]$ associated with a graph is a symmetric matrix with $a_{ij}=1$ if $(i,j)\in E$, and $a_{ij}=0$ otherwise, for $i,j\in V$. The degree of a vertex is the number of edges incident to it, and the degree matrix for a graph on 
$n$ vertices is defined as $D=\diag(\mbox{deg}_1,\hdots,\mbox{deg}_n),$ where $\mbox{deg}_i$ is the degree of $i\in V.$ Then the matrix $L=D-A$ represents the discrete Laplacian operator associated with the graph.    

In \cite{HZLJ20, ZHJL2024, corrales2018arithmetical, BHDNJC18}, the investigation of arithmetical structures on complete graphs, paths, and cycles is conducted from a multi-directed graph perspective, and Lorenzini's result of finite number of arithmetical structures on a graph is further established. The study also reveals that the number of arithmetical structures on a path corresponds to the Catalan number. Similarly, \cite{BHDNJC18} demonstrates that for paths and cycles, the enumeration of arithmetical structures is governed by the combinatorial properties of Catalan numbers. Given the complexity of identifying arithmetical structures on general graphs, the literature emphasizes the characterization of these structures on specific graphs. The article \cite{BHDNJC18} examines arithmetical structures on graphs with a cut vertex. On complete and star graphs, arithmetical structures are shown to correspond one-to-one with a variant of Egyptian fractions. For instance, arithmetical 
$\d$-structures on the star graph 
$K_{n,1}$ can be represented as positive integer solutions to the equation 
$d_0=\sum_{i=1}^n 1/d_i$ . This solution is often referred to as an Egyptian fraction representation of 
$d_0$ \cite{Sloane18}. A related problem, with the additional constraints 
$d_0=1$ and 
$d_1\leq \hdots \leq d_n,$, was studied by Sándor \cite{sandor03}, who provided upper and lower bounds for the number of solutions, with the upper bound later refined by Browning and Elsholtz \cite{TC11}. Despite these efforts, the gap between the lower and upper bounds remains large, leaving the asymptotic growth uncertain. Recent works \cite{BHDNJC18, KAAL20} have expanded the counting of arithmetical structures to various families of graphs, including path graphs, cycle graphs, bidents, and certain path graphs with doubled edges.

It should be noted that the equation (\ref{eqn:vcond}) can be written as a matrix equation \begin{equation}\label{eqn:mcond}
    (\diag(\d)-A)\r=0,
\end{equation} which is satisfied by considering $\d$ as the degree vector associated with the graph and $\r$ is the all-one vector, which follows from the fact that all-one vector is an eigenvector of the Laplacian matrix $L$ corresponding to the eigenvalue zero. In general, given $\d,$ the $\r$-structure on a graph is the integer solution set of the equation (\ref{eqn:mcond}). Considering the operator $L(G,\d)=\diag(\d)-A$ as a generalization of the Laplacian matrix $L,$ and given $(\d,\r)\in \mathcal{A}(G),$ \begin{equation}\label{eqn:defcg}
    K(G,\d,\r) = \ker(\r^T) / \Im(L(G,\d)),
\end{equation} defines the critical group of $(G,\d,\r),$ which generalizes the concept of critical group of $G$ \cite{biggs1997algebraic}. It should be noted that critical group is also known as sandpile group \cite{LLJP2010, Dhar90, PCK88}, and the elements of the critical
group represent long-term behaviours of the well-studied abelian sandpile model on $G$ \cite{BTW88, LP10, criticalgroupon(starandcompletegraph)}.

  \begin{figure}[t]  
\centering
\begin{tikzpicture}[scale=2, every node/.style={circle, draw, minimum size=0.5cm}]
    \node [fill=blue!20] (v1) at (90:1) {$v_1$};
    \node [fill=blue!20] (v2) at (150:1) {$v_6$};
    \node [fill=blue!20] (v3) at (210:1) {$v_5$};
    \node [fill=blue!20] (v4) at (270:1) {$v_4$};
    \node [fill=blue!20] (v5) at (330:1) {$v_3$};
    \node [fill=blue!20] (v6) at (30:1) {$v_2$};

    \node[fill=blue!20] (center) at (0,0) {$v_0$};

    \draw (v1) -- (v2) -- (v3) -- (v4) -- (v5) -- (v6) -- (v1);

    \foreach \i in {1,2,3,4,5,6} {
        \draw (center) -- (v\i);
    }
\end{tikzpicture}
\caption{Wheel graph $W_6$}
\label{Fig:w6}
\end{figure}
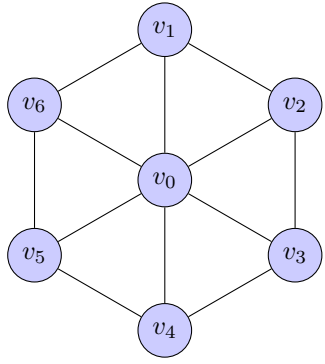

In this paper, we consider characterizing the arithmetical structure on the wheel graphs. The wheel graph  with $n$ outer vertices is denoted as $W_n$. With a total $n+1$ vertices and $2n$ edges, $W_n$ is formed by connecting a single central vertex (labeled $v_0$) to all vertices of an outer cycle $C_{n}$ with vertices labeled as $v_1,\hdots, v_{n}$. For instance, $W_6$ is given by Figure \ref{Fig:w6}. Obviously, $W_1$ is an edge and $W_2$ is isomorphic to $C_3.$ Thus we consider $W_n$ to be nontrivial when $n\geq 3.$ The adjacency matrix $A(W_n)=[a_{v_i,v_j}]$ of order $n+1$ is a symmetric matrix is given by \begin{equation}\label{eqn:adjmwn}
    a_{v_i,v_j}=\begin{cases}
        1, \,\, \mbox{if} \,\, i=0, j=1,\hdots,n, \\
        1, \,\, \mbox{if} \,\, j=i+1, i= 1,\hdots,n-1, \\
        0, \,\, \mbox{otherwise}.
    \end{cases}
\end{equation} Wheel graphs can be described in multiple ways employing the notions of the star graph $S_n$ on $n$ vertices, the cycle graph $C_n$, and a singleton vertex graphs, denoted as $K_1$ (the complete graph on one vertex). Indeed, $W_n=C_n+K_1,$ where $+$ denotes the join of two graphs \cite{diestel2024graph, godsil2001algebraic}. Besides, $W_n$ can be formed from $S_n$ with an additional set of edges is added by connecting all the outer nodes in a cycle.  
It will be interesting to investigate and explore whether any one of $\A(C_n),$ $\A(S_n),$ and $\A(W_n)$ can characterize the others. As mentioned above, $\A(C_n)$ is controlled by the combinatorics of Catalan numbers $(2n)!/(n+1)!\, n!$ \cite{BHDNJC18}.



We denote the all-one vector of dimension $n$ as ${\bf 1}_n.$ The contribution of this paper are as follows. \begin{enumerate}
  \item If $(\d,\r)\in \mathcal{A}(W_n),$ $n\geq 4$ with $\d=(d_0,d_1,\hdots,d_{n}),$ $\r=(r_0,r_1,\hdots,r_{n})$ and if $\r \neq {\bf 1}_{n+1}^t$ then one of the following holds:
\begin{enumerate}
    \item $ d_0> n$ and $d_i<3,$ for some $i>0$.
    \item $d_0< n$ and $d_i>3$, for some $i>0$.
    \item $d_0=n$ and $d_i<3$, $d_j>3,$ for some $0< i\neq j\le n.$
\end{enumerate}
If $d_i=1$ for some $i$,  then $\textbf{d} _u >1$  for all $u \in N_{W_{n}}(v_i)$, where $N_{W_{n}}(v_i)$ is the neighbourhood of $i$th vertex $v_i$ in $W_{n}$.

    \item The paper demonstrates how the arithmetical structure on $W_{n+1}$ is derived from the arithmetical structure on $W_{n}$. (Section~4)

    \item  It describes the arithmetical structures on $W_n$ derived from the cycle graph $C_n$. (Section 3.1)

    \item If $(\d,\r)\in \A(C_n)$ and $p=l.c.m. \{r_i: \ 1 \leq i \leq n\}$.  If $\tilde{r}_0\in \mathbb{N}$ such that $p$ divides $\tilde{r}_0$, $\tilde{r}_0$ divides $\sum_{i=1}^n r_i$ and ${\bf {\tilde{r}}} = ( \tilde{r}_0, r_1, \cdots, r_n)$,  then $\bf{\tilde{r}} \in \A_{\r}(W_n)$. 

    \item  Inspired by the blowup operation introduced by Lorenzini in [\cite{LD89}, page 485] and the clique–star transformation described in \cite{corrales2018arithmetical}, we define several analogous operations in different directions. Like the blowup operation, the operations introduced in this article, when applied to a graph, do not always result in a graph. (Section~4)


\end{enumerate}
 




\section{Preliminaries}

In this section, we briefly review some elementary matrix theoretic notions that will be used in sequel. For more information on these concepts, refer to \cite{Giorgi, corrales2018arithmetical, BHDNJC18}.

\begin{definition}
A square matrix $A$ is called reducible if there exists a permutation matrix $P$ such that $$PAP^{t} = \begin{pmatrix}
    A_1 & * \\
    0 & A_2 \\
\end{pmatrix}$$ for some non-trivial square matrices $A_1$ and $A_2.$  A square matrix $A$ which is not  reducible  is called irreducible.  
\end{definition}
We recall following Theorem from \cite{Giorgi}
\begin{theorem}\label{reducible}
The square matrix $A$ of order $n$ is reducible if and only if there exists a partition
$N_1$, $N_2$, $N_1\ne \emptyset $, $N_2\ne \emptyset$, of the set $N= \{ 1, \ldots, n\}$ such that 
$$ i\in N_1, j\in N_2\implies a_{ij}=0.$$
\end{theorem}
\begin{definition} [Neighbourhood of a vertex $v$ in a graph $G$] \cite{diestel2024graph}
 The neighbourhood of a vertex $v$ in a graph $G$ is the subgraph of $G$ induced by all vertices adjacent to $v$, i.e., the graph composed of the vertices adjacent to $v$ and all edges connecting vertices adjacent to $v$.  The neighbourhood is often denoted by $N_G(v)$.
\end{definition}

Here we recall the classical concept of an $M$-matrix and study a class of $M$-matrices whose proper principal minors are positive and its determinant is non-negative. Let us begin with some definitions:

\begin{definition}[Real non-negative matrix] \cite{corrales2018arithmetical}
A real square matrix is called non-negative if all its entries are non-negative real numbers.    
\end{definition}

\begin{definition}[$Z$-matrix] \cite{corrales2018arithmetical}
 A real matrix $ A = (a_{i,j}) \in \mathbb{R}^{n \times n}$ is called a Z-matrix if $a_{i,j} \leq 0 ,\ for \ all \ i \neq j$. 
\end{definition}

\begin{definition}[$M$-matrix] \cite{corrales2018arithmetical}
 A $Z$-matrix $A$ is an $M$-matrix if there exists a non-negative matrix $N$ and a non-negative number $\alpha$ such that such that $A = \alpha I- N \ and \ \alpha \geq \rho(N)$, where $\rho(N) = max\{|\lambda|  : \lambda \in \sigma(N)\}$.   
\end{definition}
The study of $M$-matrices can be divided into two major parts: non-singular M-matrices (see[\cite{Berman1994andPlemmons},Section 6.2])

\begin{definition}\label{almost-non}
A real matrix $A =(a_{i,j})$ is called an almost non-singular $M$-matrix if $A$ is a $Z$-matrix, all its proper principal minors are positive and its determinant is non-negative.    
\end{definition}
Next, we recall following Theorems:
\begin{theorem} \cite{corrales2018arithmetical}\label{Theorem 2.6 from Corrales and Valencia}
 If $M$ is a real $Z$-matrix, then the following conditions are equivalent: 
 \begin{itemize}
     \item [(1)] $M$ is an almost non-singular $M$-matrix.
     \item [(2)] $M +D$ is a non-singular $M$matrix for any diagonal matrix $D>_{\neq} 0$. 
     \item [(3)] $\det(M) \geq 0$ and $\det(M +D) {\gneq} \det(M +D'){\gneq} 0$  for any diagonal matrices $D >_{\neq} D'{\gneq} 0.$  
 \end{itemize} 
\end{theorem}
\begin{theorem} \cite{corrales2018arithmetical}\label{Theorem 3.2 from Corrales and Valencia}
Let $M$ be a $Z$-matrix. If there exists ${\bf r}$ with all its entries positive such that $M{\bf r}^{t} = 0^t,$ then $M$ is an $M$-matrix. Moreover, $M$ is an almost non-singular $M$-matrix with $\det(M) = 0$ if and only if $M$ is irreducible and there exists ${\bf r}$ with all its entries positive such that $M{\bf r}^t = 0^t. $   
\end{theorem} 

\section{Some Properties of Arithmetical Structures on Wheel Graphs}

In this section, we derive certain necessary and sufficient conditions for the entries of the pair of vectors $(\d,\r)$ such that it is an element of $\A(W_n).$ First, we observe if $\r=(r_0,r_1,\hdots,r_{n})$ is an  $\r$-structure of $W_n$ then from equation (\ref{eqn:vcond}) the following relations hold. \begin{eqnarray}\label{rstructure}
\begin{cases}
    r_0\, \, \mbox{divides} \,\, \sum_{j=1}^{n} r_j \\
  r_i \;\; \mbox{divides} \;\; r_0+r_{i-1}+r_{i+1},\;\; \forall i\in [n]
\end{cases}
\end{eqnarray}  Conversely, if relations given by equation (\ref{rstructure}) are satisfied  and $gcd\{r_i: 0\le i\le n\}=1$ then  $\r=(r_0,r_1,\hdots,r_{n})$ is an arithmetical  $\r$-structure on $W_n$.   In particular, for the degree vector $\d$ of $W_n$ with $d_0=n-1,$ $d_j=3$, $j=1,\hdots,n$ and $\r={\bf 1}_n$, the all-one vector of dimension $n$, $(\d,\r)\in \A(W_n).$ This follows from the fact that the all-one vector is an eigenvector corresponding to the Laplacian eigenvalue zero.

Before moving to next result we will fix some notations. If $\textbf{d} , \textbf{a} \in $ $\mathbb{R}^V $ then, we say that $\textbf{d} \leq \textbf{a}$ if and only if $\textbf{d}_v \leq \textbf{a}_v$ $\forall \, v \in V$, where $\leq$ is a partial order in $\mathbb{R}^V$.
For $\textbf{d} , \textbf{a} \in $ $\mathbb{N}^V $  we say that $\textbf{d} < \textbf{a}$ if and only if $\textbf{d} \leq \textbf{a}$ and $\textbf{d} \neq \textbf{a}$. Now, we have the following theorem:


\begin{theorem} \label{Theorem for unique Laplacian Arith Structure on Wheel graph}
Let $W_{n}$ be the wheel graph and $\textbf{(d,r)}\in \A(W_n).$  If $\r \neq {\bf 1}_{n+1}^t$ then one of the following holds:
\begin{enumerate}
    \item $ d_0> n$ and $d_i<3,$ for some $i>0$.
    \item $d_0< n$ and $d_i>3$, for some $i>0$.
    \item $d_0=n$ and $d_i<3$, $d_j>3,$ for some $0< i\neq j\le n.$
\end{enumerate}
If $d_i=1$ for some $i$,  then $\textbf{d} _u >1$  for all $u \in N_{W_{n}}(v_i)$, where $N_{W_{n}}(v_i)$ is the neighbourhood of $i$th vertex $v_i$ in $W_{n}$.

\end{theorem}
\begin{proof}
Let $W_{n}$ be a wheel graph with vertices   $v_0, v_1, v_2,  \ldots v_n $,   where $v_0$ is the central vertex with degree $n$ and $v_1, v_2 \ldots, v_n$ are the vertices of cycle with degree $3$ each and $ v_{n+1} = v_n$. Since $\textbf{(d,r)} \neq ((n, 3{\bf 1}_{n-1}^t), {\bf 1}_{n}^t) \in \A 
  (W_{n-1})$ and $0 < \textbf{d}$, by using Theorem (\ref{Theorem 2.6 from Corrales and Valencia}) and Theorem (\ref{Theorem 3.2 from Corrales and Valencia}), we get that neither $ ((n),3,3,3, \cdots,3) \geq \textbf{d}$ nor  $ \textbf{d}\geq((n),3,3,3, \cdots,3)$. From this we conclude that  one of the following holds: 
  \begin{enumerate}
    \item $ d_0> n$ and $d_i<3,$ for some $i>0$.
    \item $d_0<n$ and $d_i>3$, for some $i>0$.
    \item $d_0=n$ and $d_i<3$, $d_j>3,$ for some $0< i,j\le n.$
\end{enumerate}
  Now as $L(W_{n},\textbf{d}) \textbf{r}^t = \textbf{0}^t$, we get 
  \begin{equation} \label{nmt1}
  \begin{aligned}
    r_0d_0 &= r_1 + r_2 + \cdots + r_n \\
    d_ir_i &= r_0 + r_{i-1} + r_{i+1}, \forall i\in [n] \\
    \end{aligned}
    \end{equation}
Let  $d_i=1$ for some $i\ge 0$.  We have $N_{W_{n}}(v_i)=\{ v_{i-1}, v_{i+1}, v_0\}$  $\forall i\in [n]$,   $N_{W_{n}}(v_i)=\{ v_1, v_2, \ldots, v_n\}$ for $i=0$.  Suppose that  $\textbf{d} _u =1$  for some  $u \in N_{W_{n}}(v_i)$. Using Equation (\ref{nmt1}) this gives  a contradiction.
\end{proof}


The following propositions provide  a class of elements of $\A(W_n).$

\begin{proposition}
There is an arithmetical structure $(\textbf{d}, \textbf{r})$ on $W_n$ such that $\textbf{r} = (n, 1, 1, \ldots, 1). $.    
\end{proposition}

\begin{proof}
Since $\textbf{r}$ satisfies (\ref{rstructure}), we have thus produced an arithmetical structure on $W_n$.    
\end{proof}

\begin{proposition}  If $k$ divides $n$, then $\textbf{(d,r)}\in \A(W_n)$, where  $\textbf {d} = (\frac{n}{k},2+k,\cdots,2+k)$ and $\textbf{r} = (k,1,\cdots,1).$ 
\end{proposition}

\begin{proof}
 This follows because $\textbf{d} \in \mathbb{N_+}^n$, $\textbf{r} \in \mathbb{N_+}^n$, the wheel graph satisfies
$$\begin{pmatrix}
\cfrac{n}{k} & -\textbf{1}_n\\
 -\textbf{1}_{n}^t & L(C_n,(2+k) \textbf{1}_n)
 \end{pmatrix} \begin{pmatrix}
k\\
 \textbf{1}_n
 \end{pmatrix} = \bf{0},$$\\
where $L(C_n,(2+k) \textbf{1}_n) = \diag (2+k) - L(C_n).$
\end{proof}

%

\begin{proposition}
Let $C_n$  denote the cycle graph on $n$ vertices and $\textbf{(d,r)} \in \mathbb{N_+}^n \times \mathbb{N_+}^n$ be such that $L(C_n,\d)\r = a {\bf 1}_n$ with $a$ divides $\sum _{i=1} ^ n r_i$. If $ g = gcd(a,r_1,\cdots,r_n)$ then, $(\Tilde{\bf d},\Tilde{\bf r})\in \A(G)$, where 
 $$ \Tilde{\bf d} = \left(\frac{\sum _{i=1} ^ n r_i}{a},d_1,\cdots,d_n\right) \ and \ \Tilde{\bf r} = \left(\frac{a}{g},\frac{r_1}{g},\cdots,\frac{r_n}{g}\right).$$
\end{proposition}
\begin{proof}
This follows because  $\Tilde{\bf d} \in \mathbb{N_+}^{n+1}$ , $\Tilde{\bf r} \in \mathbb{N_+}^{n+1}$ and 
$$ \begin{pmatrix}
   \cfrac{c}{a} & -\bf{1_n}\\
   -\bf{1_n}^t  & L(C_n,\bf {d})
\end{pmatrix} \begin{pmatrix}
 \cfrac{a}{g}\\
 \cfrac{\bf{r}}{g}
\end{pmatrix} = \bf{0}.$$ Then the proof follows.
\end{proof}


\subsection{$\A(W_n)$ from $\A(C_n)$}


Consider a cycle graph $C_n$ on $n$ vertices. Let $\r = (r_1, \ldots, r_n)$ be a primitive positive integer vector. Then $\textbf{r}\in \A_{\r}(C_n)$ if and only if $r_i|(r_{i-1} + r_{i+1}), i \in \{1,\hdots,n\}$ with the indices taken modulo n \cite{BHDNJC18}.
It follows from \cite{BHDNJC18} that for each $k \in [n],$ the number of arithmetical structures $(\d, \r)$ on $C_n$ with $\textbf{r}(1) = k,$ the number of $1$'s in $\r,$
is $\begin{pmatrix}
    2n-k-1\\
    n-k
\end{pmatrix}$
and 
$|\A(C_n)| = \begin{pmatrix}
    2n-1\\
    n-1
\end{pmatrix}.$

The next theorem describes an arithmetical structure on $W_n$ derived from arithmetical structure on cycle graph $C_n$.

\begin{theorem} \label{arithCntoWn} 
If $(\d,\r)\in \A(C_n)$ such that $r_i$ divides $\sum_{j=1 \atop j\neq i}^{n} r_j$, for all $1\le i\le n$,  then $W_n$ has an arithmetical structure $(\tilde{\textbf{d}},\tilde{\textbf{r}})$ given by 
 $$  \tilde{d_i}  = \begin{cases}
         1 & \text{if}\;i = 0\\
           d_i + \frac{\sum_{j=1}^{n}r_j}{r_i} & \text{if}\;  1\leq i\leq n\\    
   \end{cases}$$
   $$\tilde{r_i}  =  \begin{cases}
        r_1+r_2+\cdots+r_{n} &\text{if}\;  i =0\\
        r_i & Otherwise \\
   \end{cases}
$$
  
\end{theorem}
\begin{proof}
    Let $A$ denote the adjacency matrix of $C_{n}$. Observe that adjacency matrix of $W_n$ is given by 
   $$\tilde{A}= \begin{pmatrix}
    0 & {\bf 1}_{n}^t\\
    {\bf 1}_n & A
\end{pmatrix},$$
By direct computation one can verify that $(\diag(\Tilde{\textbf{d}}) - \Tilde{A} )\tilde{\textbf{r}} = 0$. Hence the result follows.
\end{proof}

\begin{example}
Consider the cycle $C_3$. Then the arithmetical structure $(\d, \r)$ on $C_3$ is given by 
\begin{center}
    \begin{tabular}{ | l | l | l| p{5cm} | }
    \hline
   $\d = (d_1, d_2, d_3)$ & \r= $(r_1, r_2, r_3)$  & \#\\
   \hline
(2, 2, 2) & (1, 1, 1) & 1 \\
(3, 3, 1) & (1, 1, 2) & 3 \\
(5, 2, 1) & (1, 2, 3) & 6 \\
     \hline
    \end{tabular}
\end{center}
Now, by the Theorem~\ref{arithCntoWn}, the below table gives some arithmetical structures $(\tilde{\d}, \tilde{\r})$ of $W_3$ are  follows:
\begin{center}
    \begin{tabular}{ | l | l | l| p{5cm} | }
    \hline
   $\tilde{\d} = (\tilde{d}_0, \tilde{d}_1, \tilde{d}_2, \tilde{d}_3)$ & $\tilde{\r}= \tilde{r}_0, \tilde{r}_1, \tilde{r}_2,\tilde{r}_3)$  & \#\\
   \hline
(1, 5, 5, 5) & (3, 1, 1, 1) & 4 \\
(1, 7, 7, 3) & (4, 1, 1, 2) & 12 \\
(1, 11, 5, 3) & (6, 1, 2, 3) & 24 \\
     \hline
    \end{tabular}
\end{center}
\end{example}

We now define an action of $\mathbb{Z}_n$ on $\A_{\r}(W_n).$ This action provides many arithmetical structures on $W_n$. 

\begin{theorem}\label{action}
There exists an action of $\mathbb{Z}_n=\{1,2,\hdots,n\}$ on the set $\A_{\r}(W_n).$ 
\end{theorem}

\begin{proof}
 Denote the set of all bijective functions from $\A_{\r}(W_n)$ to itself by $Iso(\A_{\r}(W_n))$. We define an   action $\rho$ of cyclic group $\mathbb{Z}_n$ on the set  $\A_{\r}(W_n)$, that is, a group homomorphism  $\rho:\mathbb{Z}_n\to Iso(\A_{\r}(W_n))$  by
$$ \rho(c)(r_0, r_1, \cdots, r_n) = (r_0, r_{c+1},\cdots,r_n,r_1,\cdots,r_c)=(r_0,r_{c\oplus 1},r_{c\oplus 2}, \hdots, r_{c\oplus n}),$$ where $\oplus$ is modulo $n$ addition for any $c\in \mathbb{Z}_n$  and $ (r_0, r_1, \cdots, r_n) \in \A_{\r}(W_n)$. 
By using the relations \ref{rstructure}, it is clear that $\rho$ is well defined action.
\end{proof}



 \begin{example}
With respect to $\rho $ action of $\mathbb{Z}_3$ on the set $ Arith(W_3)$ the  $\rho$-orbit of arithmetical $r$-structure  $(1,6,2,3)$  on $W_3$ is given by
$$\{(1,6,2,3), (1,2,3,6), (1,3,6,2)\}.$$
 \end{example}
 
\begin{theorem}
 Let $(\d,\r)\in \A(C_n)$ and $p=l.c.m. \{r_i: \ 1 \leq i \leq n\}$.  If $\tilde{r}_0\in \mathbb{N}$ such that $p$ divides $\tilde{r}_0$, $\tilde{r}_0$ divides $\sum_{i=1}^n r_i$ and ${\bf {\tilde{r}}} = ( \tilde{r}_0, r_1, \cdots, r_n)$,  then $\bf{\tilde{r}} \in \A_{\r}(W_n)$. 
\end{theorem}

\begin{proof}
  Since $\textbf{(d,r)}  \in \A(C_n),$  we have $d_ir_i = r_{i-1} + r_{i+1}$, $i \in [n]$. Since $p|\tilde{r}_0,$  we have $r_i|\tilde{r}_0,$ $\forall 1\le i\le n.$ This implies that, for $1\le i\le n,$ 
  $$ \tilde{d}_i{r}_i  = \tilde{r}_0  + r_{i-1} + r_{i+1}, \,\,\, \forall i \in [n],$$ where, $\tilde{d}_i=\frac{\tilde{r}_0}{r_i}+ d_i.$  
  Since $\tilde{r}_0|\sum_{i=1}^nr_i$, 
      $$\tilde{d}_0\tilde{r}_0 = r_1 + \cdots + r_n,$$ for some $\tilde{d}_0\in \mathbb{N}.$
  Using the relations \ref{rstructure}, we conclude that $\bf{\tilde{r}}$ is an arithmetical $r$-structure on $W_n$.
\end{proof}

\begin{example}
 We will construct arithmetical r-structure on $W_3$ from arithmetical r-structure (1,3,2) on $C_3$. \\
 Here $r_1 = 1$, $r_2 = 2$, $r_3 = 3$. Put $\tilde{r}_0 = lcm(r_1 , r_2 , r_3)  = 6$. Hence (6,1,2,3) is an arithmetical r-structure on $W_3$.
\end{example}

\begin{note} We note that $(1,2,3,4)$ arithmetical $r$-structure   on $C_4$.  But above theorem is not applicable to it because its hypothesis is not satisfied by this. 
\end{note}

In \cite{corrales2018arithmetical}, the arithmetical structures on the clique–star transform of a graph has been studied. Given a graph $G$ and a clique $C$ (a set of pairwise adjacent vertices) of $G$, the clique–star transform of $G$, denoted by $(G, C),$ is the graph obtained from $G$ by deleting all the edges between the vertices in $C$ and adding a new vertex $v$ with all the edges between $v$ and the vertices in $C$.  recall following  relationship between For every  arithmetical structures $(\d,\r)$ on $G$ there exists an arithmetical structures $(\tilde{\d}, \tilde{\r})$ on 
cs$(G, C)$ given by \cite{corrales2018arithmetical}
\begin{equation} \label{csdu}
\tilde{d}_{u} = cs(\d, C)_{u}= \begin{cases}
    d_{u} & if u \notin C \\
    d_u +1 & if u \in C \\
    1 & if u =v \\
\end{cases}, \,\,\,
    \tilde{r}_{u} = cs(\r, C)_{u}= \begin{cases}
    r_{u} & if u \in V \\
    \sum_{u \in C} r_u & if u =v. \\
\end{cases}
\end{equation}
Now, consider the wheel graph $W_3$ and consider the clique $\{v_1, v_2, v_3, v_4.\}$.  Now,  apply the clique-star transformation on $W_3. $ Then we have a star graph $S_4$ with $5$ vertices say, $v_0, v_1, v_2, v_3, v_4.$ 
\begin{figure}[H]
\centering
\begin{tikzpicture}[scale=2, every node/.style={circle, draw, minimum size= 0.5cm}]
    \node [fill=blue!20] (v1) at (90:1) {$v_2$};
    \node [fill=blue!20] (v2) at (-30:1) {$v_3$};
    \node [fill=blue!20] (v3) at (210:1) {$v_4$};

    \node[fill=blue!20] (center) at (0,0) {$v_1$};

    \draw (v1) -- (v2) -- (v3) -- (v1);

    \draw (center) -- (v1);
    \draw (center) -- (v2);
    \draw (center) -- (v3);
\end{tikzpicture}

\caption{\label{fig: Wheel Grap} Wheel Graph $W_3$ with 4 vertices.}
\begin{tikzpicture}
\Vertex[label=$v_4$]{a}
\Vertex[label=$v_3$,x=3]{b}
\Vertex[label=$v_2$,y=3]{c}
\Vertex[label=$v_1$,y=2]{e}
\Vertex[label=$v_0$,x=3,y=2]{d}
\Edge(e)(d)
\Edge(c)(d)
\Edge(b)(d)
\Edge(a)(d)
\end{tikzpicture} 
\caption{  Star Graph $S_4$ with 5 vertices obtained from wheel graph $W_3$ by applying clique star transformation.}
\end{figure}

Then, $(D- A(S_4))\textbf{r} =0 $ implies
\begin{align*}
 d_{i}r_i = r_0, i=1, 2, 3, 4\\
 d_{0} r_0 = r_1+r_2+r_3+r_4. 
\end{align*}
That is $r_i| r_0$ and $r_0| r_1+r_2+r_3+r_4.$ Now, we get $d_0 = \frac{1}{d_1}+ \frac{1}{d_2}+\frac{1}{d_3}+\frac{1}{d_4}. $ Note that the numbers of solutions for $d_0 = \sum_{i=1}^{n} \frac{1}{d_i}$  for $n \leq 8$ are given by sequence A280517 in [Slo18]. The number of solutions for $n \leq 8$ is as follows: $1, 2, 14, 263, 13462, 2104021.$ So, the number of solutions for our case is $263$.\\
Next,  we have the following corollary. 
\begin{theorem}
Consider $W_3$. Then the cardinality of the set of  arithmetical structures $ (\textbf{d}, \textbf{r})$  is bounded above by $263$. 
\end{theorem}
\begin{proof}
Relations \ref{csdu} gives an injective function from the set of arithmetical structures $Arith(G)$ of a graph $G$ to the set of arithmetical structures of its clique star transform cs$(G,C)$.Since   clique star transformation of $W_3$ is $S_4$ and   the cardinality of the set of  arithmetical structures   of $S_4$  is $263, $  we conclude the result. 
\end{proof}

Next example gives a list of  $167$ arithmetical structures on $W_3.$\\



\begin{example} \label{ASW4} There are atleast $167$ arithmetical structures on wheel graph $W_3$.  Note that, if $(\d,\r)\in \A(W_3)$, then $(\pi \d, \pi \r)\in \A(W_3),$ for every permutation $\pi\in \mathcal{S}_4$. We have listed all $167$ arithmetical structures in the table given below: 
\begin{center}
    \begin{tabular}{ | l | l | l| p{5cm} | }
    \hline
   ${\bf{ d }} = (d_1, d_2, d_3, d_4)$ & \bf{r }= $(r_1, r_2, r_3, r_4)$  & $\#\{(\pi \d,\pi \r):\pi\in \mathcal{S}_4\}$ \\
   \hline
   (3, 3, 3, 3) & (1, 1, 1, 1) &  1 \\
   (5, 5, 2, 2) & (1, 1, 2, 2) &  6 \\
   (5, 5, 5, 1) & (1, 1, 1, 3) &  4 \\
   (11, 2, 2, 3) & (1, 4, 4, 3) & 12  \\
  (11, 1, 5, 3), & (1, 6, 2, 3) & 24 \\
  (7, 7, 3, 1) & (1, 1, 2, 4) & 12 \\
    (11, 11, 1, 2) & (1, 1, 6, 4) & 12   \\
  (9, 4, 4, 1) & (1, 2, 2, 5) & 12 \\
  (19,3 ,4 ,1 ) & (1, 5, 4, 10) & 24 \\
  (23,7 ,2 ,1 ) & (1, 3, 8, 12) & 24  \\
 (17, 8, 2, 1) & (1, 2, 6, 9) & 24 \\
  (5, 2, 3, 3) & (2, 4, 3, 3) &  12 \\
   \hline
    \end{tabular}
\end{center}
\end{example}


For an arithmetical $r$-structure $\r$, let $\r(1) = \# \{i: r_i=1\}. $




\begin{example}
Some of the arithmetical structures $(\bf{d}, {\bf r})$ on $W_3, W_4$, $W_5$ with $\r(1) = 4$ are given below 
\begin{center}
    \begin{tabular}{ | l | l| l|  }
    \hline
 $W_n$  & \bf{r } & \# \\
   \hline
 3   & (1, 1, 1, 1)  & 1 permu\\
 \hline
 4  & (1, 1, 1, 1, 1), (1, 1, 1, 1, 3), (2, 1, 1, 1, 1), (4, 1, 1, 1, 1) &  16 permu\\
 \hline 
 5 &  (1, 1, 1, 1, 1, 1), (1, 1, 1, 1, 1, 3), (5, 1, 1, 1, 1, 1), (2, 2, 1, 1, 1, 1)&\\
  & (3, 5, 1, 1, 1, 1), (6, 2, 1, 1, 1, 1), (7, 3, 1, 1, 1, 1)&\\
   \hline
   \end{tabular}
   \end{center}
\end{example}

\begin{example}
Some of the arithmetical structures $(\bf{d}, {\bf r})$ on $W_3, W_4$ with $\r(1) = 3$ are given below

\begin{center}
    \begin{tabular}{ | l | l| }
    \hline
 $W_n$ &   \bf{r }   \\
   \hline
 3 &   (1, 1, 1, 1), (1, 1, 1, 3)  \\
 \hline
 4 &  (1, 1, 1, 1, 1), (1, 1, 1, 1, 3), (2, 1, 1, 1, 1), (4, 1, 1, 1, 1), \\
 &(1, 1, 1, 2, 2), (1, 1, 1, 2, 4), (1, 1, 1, 6, 4) \\
   \hline
   \end{tabular}
   \end{center}   
\end{example}

\section{Generalized Blowup Operation on Arithmetical Graph}
Lorenzini in [\cite{LD89},page 485] introduced an operation, called the blowup, of an arithmetical structure of a graph. The blowup generalizes the clique–star transformation \cite{corrales2018arithmetical} of a graph. Motivated by the blowup operation introduced by Lorenzini in [\cite{LD89}, page 485] and the clique–star transformation described in \cite{corrales2018arithmetical}, we define a series of analogous operations that extend these concepts in various directions. Like the blowup operation, the operations presented in this work, when applied to a graph, do not always yield another graph.

To address this limitation, in this section, we introduce a class of matrices, denoted as $\mathfrak{M}$, whose elements, termed "generalized graphs," serve as matrix analogs of traditional graphs. We also define the notion of a generalized arithmetical structure for these generalized graphs. Importantly, we observe that the operations mentioned earlier, when applied to a graph, produce elements of $\mathfrak{M}$, effectively transforming graphs into generalized graphs.


\begin{definition}\cite{LD89}
Let $G$ be any connected graph (without loops) having vertex set $\{v_1, \dots, v_n\}$. Let  $a_{ij}$ be the number of edges joining $v_i$ and $v_j$. Let $A$ be the   $n\times n$ matrix having $(i,j)th$ entry $a_{ij}$, for $i\ne j$ and having diagonal entries $0.$ 
An arithmetical graph  $(G, M, {\bf r})$ consists of  a connected graph $G$,  $\textbf{d}, \textbf{r}\in \mathbb{N}^n$ with $lcm \{r_i:1\le i\le n\}=1$ such that $M=diag(\textbf{d})-A =L(G,\d)$ and $M{\bf r}=0.$
\end{definition}

 Given an arithmetical graph  $(G, M, \textbf{r} )$  where $ \textbf{r}^t = ( r_1, r_2, \cdots, r_n)$ and integer vector $ \textbf{q}^t = (q_1, q_2, \cdots, q_n)$ such that $$x = \sum _{i = 1} ^  {n} q_ir_i \neq 0.$$


We consider 
$$M _q = \begin{pmatrix}
    M + \textbf{q}\textbf{q}^t & -\textbf{q}\\
    -\textbf{q}^t & 1
\end{pmatrix}$$
and

$$M _{q_-} = \begin{pmatrix}
    M - \textbf{q}\textbf{q}^t  & \textbf{q}\\
    \textbf{q}^t & 1
\end{pmatrix}.$$




\begin{lemma}
 Let  $(G, M, {\bf r})$ be an arithmetical graph. Let  $\textbf{q}^t = (q_1, q_2, \cdots, q_n) \in \mathbb{Z}^n$  such that $x=\sum_{i=1}^n q_ir_i\ne 0$. Then there exist invertible matrices  $P$ and $Q$   having only integer entries such that 
 $$PM _{\textbf{q}}Q= P\begin{pmatrix}
   M + \textbf{q}\textbf{q}^t & -\textbf{q}\\
    -\textbf{q}^t & 1
\end{pmatrix}Q= \begin{pmatrix}
  M& \textbf{0}  \\
  \textbf{0}& 1
\end{pmatrix}.$$ 
\end{lemma}

\begin{proof}
Take $$P=
    \begin{pmatrix}
    I_n & \textbf{q}\\
    \textbf{0} & 1
    \end{pmatrix},$$ $$Q=
    \begin{pmatrix}
    I_n & \textbf{0}\\
    \textbf{q}^t& 1
    \end{pmatrix}.$$
    Then $$
    PM _{\textbf{q}}Q= P\begin{pmatrix}
   M + \textbf{q}\textbf{q}^t & -\textbf{q}\\
    -\textbf{q}^t & 1
\end{pmatrix}Q= \begin{pmatrix}
  M& \textbf{0}  \\
  \textbf{0}& 1
\end{pmatrix}.$$  Hence the proof.
\end{proof}

Note that $M_q (r_1,\ldots, r_n, x)^t= 0$  and  $  M_{q_-}(r_1,\ldots, r_n, x)^t=0$. The matrix $M_q$ is symmetric. When it defines an arithmetical graph this new graph is called the blow-up of $G$ with respect to $q$. The blow-up has the same critical groups as the graph $G$. However, the blowup does not always have a meaning in the context of graphs.

Next, we define a  class $\mathfrak{M}$ of matrices as a generalized matrix analog of graph  as follows.

\begin{definition}
Let $\mathfrak{M}$ be collection of all square non-negative integral matrices   $A$ of size $n$ such that 
\begin{enumerate}
   \item Diagonal entries of $A$ are zero.
       \item   There exists $\textbf{d}\in \mathbb{N}^n$ with $\textbf{d}I_n-A$ is almost non-singular with $\det(\textbf{d}I_n-A)=0.$ (see Definition \ref{almost-non})
\end{enumerate} 
We call the elements of $\mathfrak{M}$ as generalized graphs.
If ${\bf d}, {\bf r} \in \mathbb{N}^n$  be such that $gcd\{r_i: 1\le i\le n\}=1$ and $({\bf d}I_n-A){\bf r}=0,$ then we call $({\bf d}, {\bf r})$  generalized  arithmetical structure of $A\in \mathfrak{M}. $ We denote by $\A(A)$ the set of all arithmetical structures on $M.$
\end{definition}

By Theorem \ref{Theorem 3.2 from Corrales and Valencia}, it is clear that generalized  arithmetical structure exists for every $A\in \mathfrak{M}$. Also, by Theorem \ref{counting} from  \cite{corrales2018arithmetical}, we conclude that number of generalized arithmetical structures of a matrix in $ \mathfrak{M} $ is finite.

The following theorem establishes the finiteness of the set of generalized arithmetical structures of elements of $\mathfrak{M}$ (or generalized matrix analogs of graphs). This result raises the problem of determining the total number of arithmetical structures for the elements of $\mathfrak{M}.$

\begin{theorem}\label{counting}
   If $M$ is a non-negative matrix with all diagonal entries equal to zero, then $\A(M)$  is finite if and only if $M$ is irreducible.
\end{theorem}

Next,  we give some generalizations of blowup operation given by Lorenzini.
\begin{itemize}
\item Let $(G,M, {\bf r})$ be an arithmetical  graph  and $\textbf{d}$ be the corresponding d-structure. Let  $\textbf{p}, \textbf{q} \in \mathbb{N}^n$.   Define a matrix 
$$\hat{M} = B_{\textbf{p},\textbf{q}} (M) = \begin{pmatrix}
    1 & -\textbf{q}\\
    -\textbf{p}^t & \textbf{p}^t\textbf{q} +  M
\end{pmatrix}.$$  Define ${\bf \hat{d}}, {\bf\hat{r}}\in \mathbb{N}^{n+1}$, where ${\bf\hat{d}}=(\hat{d}_0, \hat{d}_1, \ldots, \hat{d}_n)^t$ , ${\bf\hat{r}}= (\hat{r}_0, \hat{r}_1, \ldots, \hat{r}_n)^t$ by 
\begin{equation} \label{gen-blowup}
\begin{aligned}
   \hat{d_i} & = \begin{cases}
        & 1, \ \ if \ i = 0\\
        & d_i + p_iq_i, \ \ Otherwise
   \end{cases}\\
   \hat{r_i} &  =  \begin{cases}
       & \sum _{j = 1} ^ {n} r_j q_j, \ \ if \ i = 0\\
       & r_i, \ \ Otherwise
   \end{cases}
\end{aligned}
\end{equation}

Clearly, $ B_{\textbf{p},\textbf{q}} (M){\bf\hat{r}}=0.$  By Theorem \ref{Theorem 3.2 from Corrales and Valencia}  and Theorem~\ref{reducible},  $B_{\textbf{p},\textbf{q}} (M)$  is an almost non-singular $M$-matrix.  We shall call this a generalized blowup of $(G,M,{\bf r})$. Note that $\diag(\hat{\textbf{d}})- B_{\textbf{p},\textbf{q}} (M)$ is a generalized graph. Its  generalized arithmetical structure is given by Equation \ref{gen-blowup}. For $\textbf{p}=\textbf{q}$, $B_{\textbf{p},\textbf{q}} (M)=M_q$.
    \item   Given a non-negative integral $n \times n$ matrix $A$ with all its diagonal entries equal to zero, $\textbf{d}, \textbf{r}\in \mathbb{N}^n $ such that $(\diag(\textbf{d})-A)r=0.$  Let $\textbf{p} = (p_1, p_2, \cdots, p_n) \in \mathbb{N}^n, \textbf{q} = (q_1, q_2, \cdots, q_n) \in \mathbb{N}^n$ with $g = gcd ( q_1, q_2, \cdots, q_n)$. Define 
$$\tilde{A} = B _{\textbf{p},\textbf{q}} (A) = \begin{pmatrix}
    0 & \textbf{q}\\
    \textbf{p}^t & -\frac {[\textbf{p}^t\textbf{q}]}{g}  + A
\end{pmatrix},$$ 

where $[\textbf{p}^t\textbf{q}]=\textbf{p}^t\textbf{q}-\diag (p_1q_1, \ldots, p_nq_n)$. 
Define  ${\bf \hat{d}}, {\bf \hat{r}}\in \mathbb{N}^{n+1}$, where ${\bf \hat{d}}=(\hat{d}_0, \hat{d}_1, \ldots, \hat{d}_n)$ , ${\bf \hat{r}}= (\hat{r}_0, \hat{r}_1, \ldots, \hat{r}_n)$ by 

\begin{equation} \label{gen-blowup1}
\begin{aligned}
   \hat{d_i} &  = \begin{cases}
        & g \ \ if \ i = 0\\
        & d_i + \frac{p_iq_i}{g} \ \ Otherwise
   \end{cases}\\
   \hat{r_i} & =  \begin{cases}
       & \sum _{j = 1} ^ {n} \frac{r_j q_j}{g} \ \ if \ i = 0\\
       & r_i \ \ Otherwise
   \end{cases}
\end{aligned}
\end{equation}

Clearly, $(\diag(\hat{\textbf{d}})-B_{\textbf{p},\textbf{q}} (A))\hat{\textbf{r}}=0.$   By Theorem \ref{Theorem 3.2 from Corrales and Valencia}  and Theorem~\ref{reducible},   $\diag(\hat{\textbf{d}})-B_{\textbf{p},\textbf{q}} (A)$  is an almost non-singular $M$-matrix with determinant $0$.  Note that $B_{\textbf{p},\textbf{q}} (A))$ is a generalized graph. Its generalized arithmetical structure  is  $(\hat{\textbf{d}}, \hat{\textbf{r}})$  given by Equation \ref{gen-blowup1}. For $\textbf{p}=\textbf{q}$, $g=1$,  $\diag(\hat{\textbf{d}})-B_{\textbf{p},\textbf{q}} (A)$ becomes a blowup.
\end{itemize}



The following theorem demonstrates how the arithmetical structure on $W_{n+1}$ is derived from the arithmetical structure on $W_{n}$.  

\begin{theorem}
Let $W_{n}$ be wheel graph with $n$ vertices $v_0, v_1, \ldots, v_n$, where $v_0$ is centre vertex.  If $({\bf d}, {\bf r})$ is an arithmetical structure on $W_{n}$ with $r_0|r_1+r_n$, $r_n|r_0$, $r_1|r_0,$ then $W_{n+1}$ has an arithmetical structure $({\bf \tilde{d}}, {\bf \tilde{r}})$ given by \\
 $$  \tilde{d_i}  = \begin{cases}
         1 & \text{if}\;i = n+1 \\
        d_i  & \text{if}\; i\ne 0,1,n, n+1\\
        d_0+\frac{r_1+r_n+r_0}{r_0}, &\text{if}\; i=0\\
        d_1+\frac{r_1+r_0}{r_1}, &\text{if}\; i=1\\
        d_n+\frac{r_n+r_0}{r_n}, &\text{if}\; i=n\\
   \end{cases}$$
   $$\tilde{r_i}  =  \begin{cases}
        r_1+r_n+r_{0}, &\text{if}\;  i = n+1\\
        r_i, & Otherwise
   \end{cases}
$$
\end{theorem}

\begin{proof}
    Let $A$ be the adjacency matrix of $W_{n}$. Observe that adjacency matrix of $W_{n+1}$ is given by 
   $$\tilde{A}= \begin{pmatrix}
    A -E_{(2,n+1)}-E_{(n+1,2)} & P^t\\
    P & 0
\end{pmatrix},$$
where $P=e_1+e_2+e_{n+1}\in \mathbb{Z}^{n+1}$and  $E_{(i,j)}$ denotes $n\times n$ matrix with $(ij)th$ entry $1$ and all other entries $0.$  By direct computation one can verify that $(\diag({\bf\Tilde{d}}) - \Tilde{A} ){\bf\tilde{r}} = 0$. Hence the result follows.
\end{proof}


\vspace{0.5cm}
\noindent{\bf Conclusion.} In this paper we derive several characteristics of the arithmetical stricture on the wheel graphs. Finally, we introduce a notion of blow-up operation for a class of matrices with an underlying graph structure.

\bibliographystyle{plain}
\bibliography{main}
\end{document}